\documentclass[
11pt,                          
draft,                         
english                        
]{article}

\usepackage[english]{babel}    
\usepackage{amsmath}           
\usepackage[utf8]{inputenc}    
\usepackage[T1]{fontenc}       
\usepackage{longtable}         
\usepackage{exscale}           
\usepackage[final]{graphicx}   
\usepackage[sort]{cite}        
\usepackage{array}             
\usepackage{wasysym}           
\usepackage[a4paper]{geometry} 
\usepackage[multiuser]{fixme}  
\usepackage{xspace}            
\usepackage{tikz}              
\usepackage{ifdraft}           
\usepackage{chairx}            
\usepackage[expansion=false    
           ]{microtype}        
\usepackage[nottoc]{tocbibind} 
 \usepackage[backref=page,      
            final=true,         
           pdfpagelabels       
           ]{hyperref}         
\usepackage[all]{xy}
\usepackage{slashed}
\usepackage{faktor}
\usepackage{amssymb}
\usetikzlibrary{matrix}
\usepackage{tikz-cd}

\usepackage{bbm}

\ifdraft{\synctex=1}{}

\fxusetheme{color}

\newcommand{\DCE}{\delta_\text{\tiny{CE}}}

\title{Characteristic (Fedosov-)class of a twist constructed by Drinfel'd}
\date{}
\author{
\textbf{Jonas Schnitzer}\thanks{\texttt{jschnitzer@unisa.it}}\\[0.5cm]
 Dipartimento di Matematica \\
  Università degli Studi di Salerno \\
  Via Giovanni Paolo II, 132\\
  84084  Fisciano (SA) \\
  Italy}

\begin{document}

\maketitle
\begin{abstract}
In a seminal paper Drinfel'd explained how to associate to every classical r-matrix for a Lie algebra $\lie g$ a twisting 
element based on $\mathcal{U}(\lie g)[[\hbar]]$, or equivalently a left invariant star product of the corresponding 
symplectic structure $\omega$ on the 1-connected Lie group G of g. In a recent paper, the authors solve the same problem by 
means of Fedosov quantization. In this short note we provide a connection between the two constructions by computing the 
characteristic (Fedosov) class of the twist constructed by Drinfel'd and proving that it is the trivial class given by $
\frac{[\omega]}{\hbar}$.
\end{abstract}

\section*{Introduction}
The notion of Drinfel'd twists goes back to the 80's in the context of \emph{Quantum Groups}, see e.g. the textbooks 
\cite{Etingof2002}, \cite{kassel1994quantum} and \cite{chari1995guide}. 
It is still an active field of research, since it is a powerful (functorial) tool used in Quantum field theory to quantize 
classical field 
theories, see e.g. \cite{AS2014}, and references therein. This application is due to the fact that a Drinfel'd twist together 
with a suitable Hopf algebra action always induces a deformation quantization. To be more precise, if a Lie algebra $\lie g$ acts on an associative $\algebra A$ by derivations via $\phi \colon \lie g\to \algebra A$, 
then a \emph{formal Drinfel'd twist} $\mathcal{F}\in 
\mathcal{U}(\lie g)\tensor \mathcal{U}(\lie g)[[\hbar]]$ induces a deformation of the associative product $\mu\colon \algebra A\times \algebra A\to \algebra A$ by 
	\begin{align*}
	a\star_\mathcal{F} b= \mu(\Phi\tensor \Phi(F)(a\tensor b))
	\end{align*}
for $a,b\in \algebra{A}[[\hbar]]$ and the extended action $\Phi\colon \mathcal{U}(\lie g)\to \Diffop(\mathcal{A})$. The properties of the twist (see Definition \ref{Def: Twist}) ensure that this product is associative again and that it is a formal deformation in the sense of Gerstenhaber \cite{Gerstenhaber}.

The existence of a Drinfel'd twist was first shown by 
Drinfel'd himself in \cite{Drinfeld} using the fact that a Drinfel'd twist is equivalent to a left-invariant star product and  
the construction of the  canonical star product on the dual of a Lie algebra, 
which is nowadays called Gutt-star product, since in \cite{Gutt} it is given a construction of it.
There are just very few explicit examples known, see e.g. \cite{GIAQUINTO1998133}, besides the \emph{Weyl-Moyal twist}, 
even though they are 
completely classified via there infinitesimal counter parts, the $r$-matrices, in \cite{Halbout2006} via Formality. 
This approach even provides 
a full  classifiaction  of Drinfel'd twists in terms of their infintesimal counterparts, the so-called \emph{r-matrices}.
Recall that two twists $\mathcal{F},\mathcal{F}'\in \mathcal{U}(\lie g)\tensor \mathcal{U}(\lie g)[[\hbar]]$ are said to be 
equivalent, if there exists a $S\in \mathcal{U}(\lie g)[[\hbar]]$ with $S=\id+\mathcal{O}(\hbar)$ and $\varepsilon(S)=1$, 
such that 
	\begin{align*}
	\mathcal{F}'=\Delta(S)^{-1}\cdot\mathcal{F}\cdot(S\tensor S),
	\end{align*}	
which induces an equivalence, i.e. an ismorphism $\Phi(S)\in \Diffop(\algebra A)$, of the induced deformations $\star_\mathcal{F}$ and $\star_{\mathcal{F}'}$ (see \cite{me} for more details). 
In \cite{me}, a construction of 
Drinfel'd twists and their classification
is provided with the help of a modified  Fedosov-construction, see \cite{Fedosov1994},  for symplectic Lie algebras. 
In fact, if the given 
Lie algebra is real the construction in \cite{me} can be seen as the original Fedosov-construction on a Lie group with a 
left-invariant connection.
Additionally,  the (invariant) classification result boils down to the one obtained in 
\cite{BBG1998}, where it is shown that the equivalence classes of left-invariant star products is isomorphic to
	\begin{align*}
	\frac{[\omega]}{\hbar}+ H^2_{\text{dR}}(G)^G[[\hbar]],
	\end{align*}
where the reference point  $\omega$ is the left-invariant symplectic structure of the Lie group and is just convention. 
Moreover,  $H_{\text{dR}}(G)^G$ is the cohomology of the left-invariant differential forms.  
The purpose of this note is now straight forward: computing the characteristic class in the sense of \cite{me} of the 
Drinfel'd twist obtained in \cite{Drinfeld}. To do so we proceed as follows: 
First, we recall the original construction of Drinfel'd and of the Gutt-star product and use the equivalence of Drinfel'd 
twists and left-invariant star products to obtain the chain of equivalences: 
	\begin{align*}
	\text{ Drinfel'd's Twist } 
	\sim \text{ invariant Kontsevich star product } \sim \text{ invariant Fedosov star product },
	\end{align*}      
where the first equivalence is provided in \cite{Dito1999} and the second equivalence is proven in \cite{Stefan}.
As a corollary we can compute the characteristic class of the Drinfel'd twist constructed by Drinfel'd in terms of the 
invariant Fedosov class of the left-invariant star product it induces. This is by construction in \cite{me} the class of the 
twist inducing it.

We are assuming the reader is familiar with the notion of star products and their different classifications due to 
\cite{Kontsevich2003} for Poisson manifolds and \cite{Fedosov1994} for symplectic manifolds, as well as the invariant 
counterparts \cite{D2005} and \cite{BBG1998}.  

\section{The Drinfel'd twist construction and the Gutt star product}
Drinfel'd twists can be defined on arbitrary Hopf algebras (in fact even just bialgebras) (see e.g. \cite{Giaquinto1998}), but throughout this note we use the following 
\begin{definition}\label{Def: Twist}
Let $\lie g$ be a Lie algebra. A (formal) Drinfel'd twist is an element 
$\mathcal{F}\in \mathcal{U}(\lie g)\tensor \mathcal{U}(\lie g)[[\hbar]]$, such that
	\begin{enumerate}
	\item $(\Delta\tensor \id )(\mathcal{F})\cdot(\mathcal{F}\tensor 1)
	=(\id\tensor\Delta)(\mathcal{F})\cdot(1\tensor \mathcal{F})$
	\item $(\varepsilon\tensor \id)(F)=(\id\tensor \varepsilon) (\mathcal{F})=1$
	\item $\mathcal{F}=1\tensor 1 +\mathcal{O}(\hbar ^2)$
	\end{enumerate}
for the usual Hopf algebra structures $(\mathcal{U}(\lie g),\cdot,i, \Delta,\varepsilon, S)$
\end{definition}

A Drinfel'd twist has a classical limit, which can be seen as the equivalent to Poisson structures in deformation 
quantization.

\begin{lemma}
Let $\lie g$ be a Lie algebra and let $\mathcal{F}\in \mathcal{U}(\lie g)\tensor \mathcal{U}(\lie g)[[\hbar]]$ be a 
Drinfel'd twist, then 
	\begin{align*}
	r:=F_1-F_1^{\mathrm{opp}}\in \Anti^2\lie g \subseteq \mathcal{U}(\lie g)\tensor \mathcal{U}(\lie g)
	\end{align*}
is a $r$-matrix, i.e. $[\![r,r]\!]=0$.
\end{lemma}

Let us now state the theorem of Drinfel'd connecting $r$-matrices with Drinfel'd twists. 

\begin{theorem}[Drinfel'd \cite{Drinfeld}]
Let $\lie g$ be a real finite dimensional Lie algebra and let $r$ be a $r$-matrix. Then there exists a Drinfel'd twist 
$\mathcal{F}\in \mathcal{U}(\lie g)\tensor \mathcal{U}(\lie g)[[\hbar]]$, such that 
	\begin{align*}
	r=F_1-F_1^{\mathrm{opp}}.
	\end{align*}
\end{theorem}

Note that Drinfel'd twists have a deep connection to deformation quantization, i.e. they induce a deformation quantization on 
any algebra the underlying Lie algebra acts on. 

\begin{theorem}
Let $\lie g$ be a Lie algebra acting as derivations on an algebra $\algebra{A}$ by $\phi\colon \lie g \to \Der(\algebra{A})$ 
and let $\mathcal{F}\in \mathcal{U}(\lie g)\tensor \mathcal{U}(\lie g)[[\hbar]]$ be a Drinfel'd twist. Then the extended 
action induced by $\phi$
	\begin{center}
	\begin{tikzcd}
	\mathcal{U}(\lie g) \arrow[r, " \Phi"] & \Diffop(\mathcal{A})
	\end{tikzcd}
	\end{center}
induces a formal deformation of the product of $\algebra{A}$ by 
	\begin{align*}
	a\star_\mathcal{F} b := \mu_\algebra{A}(\Phi\tensor\Phi(\mathcal{F})(a\tensor b))
	\end{align*}
for $a,b\in \algebra{A}[[\hbar]]$.
\end{theorem}

From now on, we focus on real Lie algebras where we can use the following 

\begin{corollary}\label{Cor: Eq Twist-Star}
Let $\lie g$ be a real Lie algebra, let $G$ be a Lie group with $\functor{Lie}(G)=\lie g$, then the following are equivalent:
	\begin{enumerate}
	\item a Drinfel'd twist $\mathcal{F}\in \mathcal{U}(\lie g)\tensor \mathcal{U}(\lie g)[[\hbar]]$ with corresponding 
	$r$-matrix $r$
	\item a left invariant star product $\star\colon \Cinfty(G)[[\hbar]]^{\tensor 2}\to \Cinfty(G)[[\hbar]]$
	with left-invariant Poisson structure $\pi\in \Secinfty(\Anti^2 TG)$, such that $\pi(e)=r$.
	\end{enumerate}
Moreover, two Drinfel'd twists are equivalent if and only if the left-invariant star products induced by them are invariantly 
equivalent. 
\end{corollary}

This allows us to forget about twists for the moment and take care about left-invariant star products on Lie groups and their 
equivalences. In the 
following we recall the construction of Drinfel'd \cite{Drinfeld},  which is based on the fact that there is a canonical star 
product on the  dual of a Lie algebra from \cite{Gutt}, which we in return also recall briefly. 
Let therefore $(\lie g,[-,-]) $
be a finite dimensional real 
Lie algebra, then there is a star product on $\lie g^*$ induced by the Poincaré-Birkoff-Witt isomorphism 
	\begin{align*}
	P_\hbar \colon \Sym^\bullet \lie g [[\hbar]] \to \mathcal{U}_\hbar(\lie g), 
	\end{align*}
where 
	\begin{align*}
	\mathcal{U}_\hbar (\lie g)=\frac{T^\bullet(\lie g)[[\hbar]]}{\langle x\tensor y -y\tensor x -\hbar[x,y]\rangle}.
	\end{align*}
and the Poincaré-Birkoff-Witt isomorphism is given by 
	\begin{align*}
	P_\hbar (X_1\vee\dots \vee X_k)=
	\frac{1}{k!}\sum_{\sigma\in S_k}X_{\sigma(1)}\bullet_\hbar \dots \bullet_\hbar X_{\sigma(k)}.
	\end{align*}
Here $\bullet_\hbar$ denotes the associative product in $ \mathcal{U}_\hbar(\lie g)$. 
We interpret $\Sym^\bullet \lie g$ as the polynomials on $\lie g^*$ and define a product 
	\begin{align*}
	f\star_G g =P_\hbar^{-1}(P_\hbar(f)\bullet_\hbar P_\hbar(g)).
	\end{align*}
It is easy to see that since the 
product is defined by differential operators, it can be extended to $\Cinfty(\lie g^*)$ and that this product is
 $\Ad$-invariant with respect to the 1-connected  Lie group $G$ integrating $\lie g$. We refer to this star product as the \emph{Gutt-star product}. Note that the semi-classical limit of this star product is the canonical (linear) Poisson structure on the dual of the Lie algebra. 

Let us now assume symplectic  Lie algebra $(\lie g, [-,-],\omega)$, i.e. a Lie algebra $(\lie g,[-,-])$ together with a non-
degenerate 2-form $\omega\in \Anti^2 \lie g^*$, such that $\DCE \omega=0$. We consider the central extension adapted to 
$\omega$ which is given by 
	\begin{align*}
	\lie h =\lie g \oplus \mathbb{R} \text{ with } [(X,x),(Y,y)]_\lie h=([X,Y],-\omega(X,Y)) 
	\end{align*}
and the canonical projection $ p\colon \lie h\to \lie g$, which is a lie algebra map.
	
A short computation shows that that the Jacobi identity of $[-,-]_\lie h$ is equivalent to $\DCE\omega=0$. 
Note that in $\lie h^*$, we have the distinguished element 
	\begin{align*}
	C\colon \lie h\ni (X,x)\to x\in \mathbb{R},
	\end{align*}
where we one can show that $\DCE C =-p^*\omega$. Note that $(\lie h, C)$ is a contact Lie algebra and is the \emph{contactification} of the symplectic Lie algebra $(\lie g,\omega)$, see \cite{contactstuff}.
Let us prove the following
\begin{lemma}\label{Lem: DimCoAdOrb}
Let $G$ be a Lie group with Lie algebra $\lie g$. Then 
	\begin{align*}
	\dim (\Ad_G^*(\alpha))=\rank(\DCE \alpha^\sharp).
	\end{align*}	 
\end{lemma}
\begin{proof}
It is enough to prove that the tangent space of the coadjoint orbit at $\alpha$ has dimension $\rank(\DCE \alpha^\sharp)$. 
But since it is an orbit of a Lie group action it is spanned by the fundamental vector fields. But they are given by 
	\begin{align*}
	X_\lie g(\alpha)=\frac{\D}{\D t}\At{t=0}\Ad^*_{\exp(tX)}\alpha=\frac{\D}{\D t}\At{t=0} \exp(t\ad^*_X)\alpha.
	\end{align*}
which is $0$ if and only if $\ad_X^*\alpha=0$. Moreover, we can write $\ad_X^*\alpha=-\DCE\alpha^\sharp(X)$ and the claim is 
proven.  
\end{proof}

We have additionally a Lie algebra action of $\lie g$ on $\lie h^*$ by
	\begin{align*}
	\acts \colon \lie g \ni X\to (X,0)_{\lie h^*}\in \Secinfty(T\lie h^*)
	\end{align*}
where $(X,0)_{\lie h^*}$ is the fundamental vector field of the coadjoint representation of $H$ (the 1-connected Lie group 
integrating $\lie h$). It is easy to check that 
	\begin{align*}
	[X,Y]\acts=-[X\acts,Y\acts]
	\end{align*}
and hence it is in fact a Lie algebra action. Moreover its fundamental vector fields have complete flow and hence this Lie 
algebra action lifts to a left Lie group action $\Phi\colon G\times \lie h^*\to \lie h^*$ of $G$ on $\lie h^*$ by the Theorem 
of Palais. Note that, we can lift every map in order to obtain the following commutative diagram 
	\begin{center}
	\begin{tikzcd}\label{Diag: Liegrp}
	H \arrow[r, "\Ad^* "]\arrow[d, two heads]& \group{Gl}(\lie h^*)\\
	G\arrow[ur, "\Phi "']
	\end{tikzcd}
	\end{center}
which is just the integration of the diagram 
	\begin{center}
	\begin{tikzcd}\label{Diag: Liealg}
	\lie h \arrow[r, "\ad^* "]\arrow[d,"p", two heads]& \End(\lie h^*)\\
	\lie g\arrow[ur,]
	\end{tikzcd}
	.
	\end{center}

\begin{lemma}\label{Lem: DimOrb}
The orbit $\Phi_G(C)$ for the previously defined $C\in \lie h^*$ has dimension $\dim(\Phi_G(C))=\dim (\lie g)$.   
\end{lemma}
\begin{proof}
Using Lemma \ref{Lem: DimCoAdOrb} and the fact that $\DCE C=-p^*\omega $ and that $\omega$ is symplectic, we conclude that 
the coadjoint orbit through $C$ has dimension $\dim(\lie g)$. By the diagram \ref{Diag: Liegrp}, we get hence the claim.
\end{proof}

\begin{remark}\label{Rem: Sympl}
The previous Lemma shows that $\Phi_C\colon G\ni g\mapsto \Phi_g(C) \in \Phi_G(C)$ is a local diffeomorphism. 
Moreover, we have that $\Phi_C\circ \ell_g =\Phi_g\circ \Phi_C$. Using the fact that $\Phi_G(C)$ is a coadjoint orbit of the 
group $H$, we know that it is a symplectic leaf of its canonical Poisson structure with invariant symplectic structure $
\Theta\in \Secinfty(\Anti^2T^* \Phi_G(C))$ and hence $\Phi_C^*\Theta$ is a left-invariant symplectic structure on $G$. A tiny 
computation shows that $\Phi_C^*\Theta(e)=\omega$ and hence $\Phi_C^*\Omega$ is the left translated 2-form obtained by 
$\omega$. Let us denote from now on $\Pi$ the Poisson structure induced by $\Phi_C^*\Theta$.
\end{remark}

Let us now apply these considerations to the construction of the Gutt-star product. First of all we have 

\begin{proposition}
The Gutt-star product is invariant with respect to the action $\Phi\colon G\times \lie h^*\to \lie h^*$.
\end{proposition}
Moreover, we have

\begin{lemma}
The Gutt-star product of $\lie h$ is tangential to $\Phi_G(C)$. 
\end{lemma} 

\begin{proof}
We have to check that for functions $f,g\in \Cinfty(\lie h^*)$, such that $f\at{\Phi_G(C)}=0$ we have 
$f\star_G g\at{\Phi_G(C)}=0$. But we actually show a bit more: we want to show that the Gutt-star product is tangential to 
the submanifold 
	\begin{align*}
	D=\{\alpha\in \lie h^*\ | \ \alpha((0,1))-1=0\}\supseteq_{\mathrm{open}} \Phi_G(C),
	\end{align*}
where the inclusion of $\Phi_G(C)$ in $D$ follows from the fact that $\Ad_h^*C(0,1)=C(\Ad_h(0,1))=C(0,1)=1$  and the fact that it is open follows from dimensional reasons (see Lemma \ref{Lem: DimOrb}).
Note that this is the zero set of the polynomial $p=(0,1)-1\in \lie h\oplus\mathbb{R}\in \Sym^\bullet\lie h$ and hence every 
function $f\in \Cinfty(\lie h^*)$ vanishing on $D$ is given by $f=\tilde{f}p$. Moreover, by the definition of the Gutt-star 
product we have that 
	\begin{align*}
	Y\star_G f=Y\cdot f
	\end{align*}
for a central element $Y\in Z(\lie h^*)$ seen as a polynomial. Hence, we have
	\begin{align*}
	(pf)\star_G g=((0,1)f)\star_G g- f\star_G g=p(f\star_G g)
	\end{align*}	 
and hence $(pf)\star_G g\at{D}=0$, since $(0,1)$ is central. And the claim follows. 
\end{proof}

\section{Kontsevich star product on $\lie h^*$}\label{Sec: KontGutt}

The dual of the Lie algebra $\lie h^*$ can be alternatively seen as a Poisson manifold and apply the the Kontsevich 
map to obtain a star product. More, precisely we apply the Kontsevich map to $\hbar \pi$, seen as a Maurer-Cartan element, 
which is invariant with respect to the adjoint representation of $H$ to obtain a $H$-invariant star product $\star_K$. In 
\cite{Dito1999} it was shown that the two star product are invariantly equivalent, i.e. there is a series of invariant 
differential operators $S=\id+\mathcal{O}(\hbar)\in \Diffop(\lie h^*)^H$, such that
	\begin{align*}
	S(f\star_G h)=(Sf)\star_K (S h),
	\end{align*}	 
which basically follows by the universal property of the universal enveloping algebra.
Moreover, we have 

\begin{lemma}
The Kontsevich star product obtained on the dual Lie algebra $\lie h^*$ is tangential to $\Phi_G(C)$ and so is $S$. 
\end{lemma}

\begin{proof}
The proof of both statements works exactly the same. It is based on the following observation: we can choose coordinates
$\{x_1,\dots,x_n,u\}$, where $\{x_i\}$ are coordinates in $\lie g^*$ and $u$ is the remaining coordinate, when we identify 
$\lie h^*=\lie g^*\oplus \mathbb{R}$. In this coordinates we have that the canonical Poisson structure $\pi$ is given by
	\begin{align*}
	\pi=\frac{1}{2}x_k \cdot C^k_{ij} \frac{\partial}{\partial x_i}\wedge  \frac{\partial}{\partial x_j}-
	\frac{1}{2}u\cdot \omega_{ij}\frac{\partial}{\partial x_i}\wedge  \frac{\partial}{\partial x_j},
	\end{align*}
where $C^k_{ij}$ are the structure constants of $\lie g$ and $\omega_{ij}$ are the components of the symplectic form 
$\omega$.
Moreover we have for the hyperplane $D=\{u=1\}$. A tangential differential operator has to be of the form 
	\begin{align*}
	\mathfrak{D}=\sum_{(I,k)}D_{(I,k)}(x,u)\frac{\partial^{|I|}}{\partial x^I}\bigg(\frac{\partial}{\partial u}\bigg)^k 
	\end{align*}
such that $D_{(I,k)}(x,1)=0$ if $k\neq 0$. It is easy to see that for our specific bivector field $\pi$ the Bidifferential 
operators $B_k(\pi)$ constructed from Kontsevich's formality map do not contain any derivatives in $u$-direction, since the 
Poisson bivector does not possess any and hence they are canonically tangential. 
Let us now turn towards the equivalence $S$, due to \cite[Thm 3]{Dito1999} this eqivalence is build up out of differential 
operators of the form (adapted to our coordinates)
	\begin{align*}
	D_r=\sum_{i_1,\dots, i_r=1}^{\dim\lie g+1} \mathrm{Tr}(\ad_{e_{i_1}}\dots\ad_{e_{i_r}})\frac{\partial}{\partial y_{i_1}}
	\cdots \frac{\partial}{\partial y_{i_r}}
	\end{align*}	
where $y_{i}=x_{i}$ for $1\leq i\leq \dim \lie g$ and $y_{\dim\lie g+1}=u$. and and the $e_i$'s are the corresponding basis 
vectors in $\lie h^*$. Note that $\ad_{e_{\dim \lie g+1}}=0$ and hence also in these differential operators no derivatives in 
the direction of $u$ appear. This concludes the proof. 
\end{proof}

If a star product is tangential to a submanifold it can be shrinked to it and one can obtain again a star product on it. But 
we have moreover that in our case:

\begin{corollary}
The star product obtained by shrinking the Kontsevich star product to $\Phi_G(C)$ is the same as the Kontsevich star product
constructed by the shrinked Poisson bivector $\pi\at{\Phi_G(C)}$ directly on $\Phi_G(C)$. 
\end{corollary}

Since the Gutt-star product, the Kontsevich-star product and the equivalence bewtween them are invariant under the action of 
$G$, we can pull the back via the local diffeomorphism in order to obtain two left-invariant star products on $G$ and a left-
invariant equivalence between them. 

Moreover, one can show that pull-back of the Kontsevich star product is oobtained by Dolgushev globalization of the 
Kontsevich formality with taking pull-back the canonical flat connection on $\lie h^*$ shrinked to $D$, which is invariant 
under the 
group action of $H$, since it acts by linear transformations. Moreover, since everything is left-invariant we see that the 
invariant Kontsevich class of the Kontsevich star product, and hence the Gutt--star product, is 
$[\hbar \Pi]_G$. The aim is now to compute the invariant Fedosov class of this star product, which is 
an element in  $\frac{[\Phi_C^*\Theta]}{\hbar}+ H_{dR}^2(G)^G[[\hbar]]
=\frac{[\omega]}{\hbar}+H_{\mathrm{CE}}^2(\lie g)[[\hbar]]$, where we identify the left-invariant de Rham cohomology with the Chevalley-Eilenberg cohomology and the fact that the invariant symplectic structure $\Phi_C^*\Theta$ is just $\omega$ under this identification by Remark \ref{Rem: Sympl}.    

\section{The Fedosov-class of the Drinfel'd twist}

A Drinfel'd Twist obtained by a symplectic structure $\omega\in \Anti^2\lie g^*$ is a left-invariant star product on a Lie 
group integrating the Lie algebra $\lie g$. The original Drinfel'd construction of a twist, is basically described above: the 
Drinfel'd twist is 
the "pulled-back" Gutt-star product.  
In \cite{me} it is shown that it is possible to modify the Fedosov-construction in 
order to get a full classification of symplectic Drinfel'd twists and the moduli space of symplectic Drinfel'd twists with 
$r$-matrix $r=\omega^{-1}$ is given by
\begin{align*}
\frac{[\omega]}{\hbar}+H_{CE}^2(\lie g)[[\hbar]].
\end{align*}
Moreover, it is easy to see that the left invariant star product $\star$ induced by a twist $\mathcal{F}$ from Corollary \ref{Cor: Eq Twist-Star} of characteristic class  $\frac{[\omega]}{\hbar}+\Omega$ has invariant characteristic class 
$I(\frac{[\omega]}{\hbar}+\Omega)$, where $I\colon \Anti^\bullet \lie g^*\to \Secinfty(\Anti^\bullet T^*G)^G$ is the canonical isomorphism of the Chevalley-Eilenberg complex to the left-invaraint de Rham complex.

\begin{theorem}
Let $(\lie g,\omega)$ be a symplectic Lie algebra and let $\mathcal{F}$ be the Drinfel'd twist constructed above. Then its characteristic class is
$\frac{[\omega]}{\hbar}$.
\end{theorem} 

\begin{proof}
From the discussion in Section \ref{Sec: KontGutt}, we know that the invariant Kontsevich class of the of the star product induced by Drinfel'd's twist 
is given by $[\hbar \Pi]_G$. To compute its invariant Fedosov class, we observe that in 
\cite{Stefan}, the authors construct an explicit equivalence between the Kontsevich-star product obtained by the 
globalization procedure from \cite{D2005} and a Poisson structure $\pi$ and the Fedosov star product obtained by the 
original Fedosov construction \cite{Fedosov1994} with the symplectic structure $\pi^{-1}$. 
If additionally a Lie group $G$ acts on the manifold acts on the manifold, it can be shown that this equivalence can be 
choosen to be $G$-invariant, if one can find a invariant connection, which is a straightforward computation.
Note that a Lie group acts always properly from the left 
on itself and hence it is possible to find an invariant connection, and with the above discussion, it is clear that the 
Fedosov star product constructed from $\Pi^{-1}$ is equivalent to the Drinfel'd twist. By the 
definition of the (invariant) Fedosov classes, the class of the Fedosov star product constructed by $\Pi^{-1}$ has Fedosov class $\frac{[\omega]}{\hbar}$ and so has hence the Drinfel'd twist.
\end{proof}

\bibliographystyle{plain}
\bibliography{references}

\begin{thebibliography}{10}

\bibitem{BBG1998}
M.~Bertelson, P.~Bieliavsky, and S.~Gutt.
\newblock {Parametrizing Equivalence Classes of Invariant Star Products}.
\newblock {\em Lett. Math. Phys.}, 46(4):339--345, 1998.

\bibitem{me}
S.~{Waldmann} C.~{Esposito}, J.~{Schnitzer}.
\newblock {A universal construction of universal deformation formulas, Drinfeld
  twists and their positivity}.
\newblock {\em Pacific J. Math.}, 291(2):319–358, 2017.

\bibitem{chari1995guide}
V.~Chari and A.N. Pressley.
\newblock {\em A Guide to Quantum Groups}.
\newblock Cambridge University Press, 1995.

\bibitem{Dito1999}
G.~{Dito}.
\newblock {Kontsevich Star Product on the Dual of a Lie Algebra}.
\newblock {\em Lett. Math. Phys.}, 48(4):307--322, 1999.

\bibitem{D2005}
V.~Dolgushev.
\newblock {Covariant and equivariant formality theorems}.
\newblock {\em Advances in Mathematics}, 191(1):147 -- 177, 2005.

\bibitem{Drinfeld}
V.G. {Drinfel'd}.
\newblock {On constant quasiclassical solutions of the Yang–Baxter quantum
  equation}.
\newblock {\em Sov. Math. Dokl.}, 28:667--671, 1983.

\bibitem{Etingof2002}
P.I. Etingof and O.~Schiffmann.
\newblock {\em Lectures on quantum groups}.
\newblock Lectures in Mathematical Physics. International Press, 2002.

\bibitem{Fedosov1994}
B.~V. Fedosov.
\newblock A simple geometrical construction of deformation quantization.
\newblock {\em J. Differential Geom.}, 40(2):213--238, 1994.

\bibitem{Gerstenhaber}
M.~Gerstenhaber.
\newblock {On the Deformation of Rings and Algebras}.
\newblock {\em Ann. Math.}, 79(1):59--103, 1964.

\bibitem{GIAQUINTO1998133}
A.~Giaquinto and J.~Zhang.
\newblock {Bialgebra actions, twists, and universal deformation formulas}.
\newblock {\em Journal of Pure and Applied Algebra}, 128(2):133 -- 151, 1998.

\bibitem{Giaquinto1998}
A.~Giaquinto and J.~J. Zhang.
\newblock {Bialgebra actions, twists, and universal deformation formulas}.
\newblock {\em J. Pure Appl. Algebra}, 128:133--151, 1998.

\bibitem{Gutt}
S.~{Gutt}.
\newblock {An Explicit -Product on the Cotangent Bundle of a Lie Group}.
\newblock {\em Lett. Math. Phys.}, 7:249--258, 1983.

\bibitem{Stefan}
S.~{Waldmann} H.~{Bursztyn}, V.~{Dolgushev}.
\newblock {Morita equivalence and characteristic classes of star products}.
\newblock {\em J. reine angew. Math.}, 662:95--163, 2012.

\bibitem{Halbout2006}
G.~Halbout.
\newblock {Formality theorem for Lie Bialgebras and quantization of twists and
  coboundary r-matrices }.
\newblock {\em Advances in Mathematics}, 207(2):617 -- 633, 2006.

\bibitem{kassel1994quantum}
C.~Kassel.
\newblock {\em Quantum Groups}.
\newblock Springer New York, 1994.

\bibitem{contactstuff}
Y.~Khakimdjanov, M.~Goze, and A.~Medina.
\newblock {Symplectic or contact structures on Lie groups}.
\newblock {\em Diff. Geom. Appl.}, 21(1):41 -- 54, 2004.

\bibitem{Kontsevich2003}
M.~Kontsevich.
\newblock {Deformation {Q}uantization of {P}oisson {M}anifolds}.
\newblock {\em Letters in Mathematical Physics}, 66:157--216, 2003.

\bibitem{AS2014}
A.~{Schenkel} P.~{Aschieri}.
\newblock {Noncommutative connections on bimodules and Drinfeld twist
  deformation}.
\newblock {\em Adv. Theo. Math. Phys.}, 18(3):513--612, 2014.

\end{thebibliography}

\end{document}